\documentclass[a4paper,12pt]{article}
\usepackage{latexsym}
\usepackage{amsmath}
\usepackage{amssymb}
\usepackage{enumerate}

\usepackage{theorem}
\newtheorem{theorem}{Theorem}[section]
\newtheorem{proposition}[theorem]{Proposition}
\newtheorem{lemma}[theorem]{Lemma}
\newtheorem{corollary}[theorem]{Corollary}

\theorembodyfont{\rmfamily}
\newtheorem{proof}{\textmd{\textit{Proof.}}}

\newtheorem{remark}[theorem]{Remark}
\newtheorem{example}[theorem]{Example}
\newtheorem{definition}[theorem]{Definition}

\newtheorem{acknowledgement}{\textmd{\textit{Acknowledgements.}}}

\makeatletter

\@addtoreset{equation}{section}
\makeatother

\newcommand{\qedd}{\hfill \Box}
\newcommand{\ve}{\varepsilon}
\newcommand{\lra}{\longrightarrow}

\newcommand{\N}{\ensuremath{\mathbb{N}}}
\newcommand{\R}{\ensuremath{\mathbb{R}}}
\newcommand{\E}{\ensuremath{\mathbb{E}}}

\newcommand{\cE}{\ensuremath{\mathcal{E}}}

\newcommand{\CAT}{\mathrm{CAT}}
\newcommand{\Prb}{\mathrm{Pr}}
\newcommand{\length}{\mathrm{length}}

\def\Lip{\mathop{\mathbf{Lip}}\nolimits}

\title{Markov type of Alexandrov spaces\\of nonnegative curvature
}
\author{Shin-ichi OHTA\thanks{
Partly supported by the JSPS fellowship for research abroad.}\\
Department of Mathematics, Faculty of Science,\\
Kyoto University, Kyoto 606-8502, JAPAN\\
\textit{e-mail}: sohta@math.kyoto-u.ac.jp}
\date{}
\begin{document}

\maketitle

\begin{abstract}
We prove that Alexandrov spaces of nonnegative curvature have Markov type $2$ in the sense of Ball.
As a corollary, any Lipschitz continuous map from a subset of an Alexandrov space of nonnegative curvature
into a $2$-uniformly convex Banach space is extended to a Lipschitz continuous map on the entire space.
\end{abstract}

\section{Introduction}\label{sc:int}%%%%%%%%%%%%%%%%%%%%%%%%%%%%%%%%%%%%%%%%%%%%%

The aim of the present article is to contribute to the nonlinearization of the geometry of Banach spaces
from the viewpoint of metric geometry.
Among them, our main object is Markov type of metric spaces due to Ball.

Markov type is a generalization of Rademacher type of Banach spaces.
Rademacher type and cotype describe the behaviour of sums of independent random variables in Banach spaces,
and these properties have fruitful analytic and geometric applications (cf.\ \cite{LT} and \cite{MS}).
Enflo \cite{E} first gave a generalized notion of type of metric spaces, which is called Enflo type now,
and a variant of Enflo type was studied by Bourgain, Milman and Wolfson \cite{BMW}.
After them, Ball \cite{B} introduced the notion of Markov type of metric spaces,
and showed its importance in connection with the extension problem of Lipschitz maps.
He showed that any Lipschitz continuous map from a subset of metric space $X$ having Markov type $2$
into a reflexive Banach space having Markov cotype $2$ can be extended to
a Lipschitz map on the entire space $X$.
Here Markov cotype of Banach spaces is a notion also introduced by Ball.
It is worthwhile to mention that how to formulate a notion of cotype for general metric spaces has been an important question,
we refer to \cite{MN3} for a recent breakthrough on this topic.
Markov type has found further deep applications in the extension problem of Lipschitz maps (\cite{NPSS}, \cite{MN1})
as well as in the theory of bi-Lipschitz embeddings of finite metric spaces or graphs
(\cite{LMN}, \cite{BLMN}, \cite{NPSS}, \cite{MN1}).

Until recently, the only known examples of spaces possessing Markov type $2$ had been Hilbert spaces
and their bi-Lipschitz equivalents.
Naor, Peres, Schramm and Sheffield \cite{NPSS} broke the situation
and showed that $2$-uniformly smooth Banach spaces and some negatively curved metric spaces have Markov type $2$
(Example \ref{ex:Ma}).
They also asked whether all $\CAT(0)$-spaces (nonpositively curved metric spaces)
have Markov type $2$ or not.
We will answer a similar question affirmatively {\it under the reverse curvature bound}.

Our main theorem asserts that Alexandrov spaces of nonnegative curvature have Markov type $2$
with a universal estimate on the Markov type constant (Theorem \ref{th:main}).
This theorem gives us first and rich examples of positively curved spaces having Markov type $2$.
As an immediate corollary by virtue of Ball's extension theorem,
any Lipschitz continuous map from a subset of an Alexandrov space of nonnegative curvature
into a reflexive Banach space having Markov cotype $2$ can be extended to a Lipschitz continuous map on the entire space
(Corollary \ref{cr:Lip}).
In particular, our estimate on the ratio of the Lipschitz constants is independent of the dimension.
Compare this with \cite{LS}, \cite{LPS} and \cite{LN}.
Our key tool is the inequality $(\ref{eq:bc*})$ in Theorem \ref{th:S} due to Sturm.

The article is organized as follows.
We briefly review the theories of linear and nonlinear types and Alexandrov spaces of nonnegative curvature
in Sections \ref{sc:Mark} and \ref{sc:Alex}, respectively.
Section \ref{sc:thm} is devoted to the proof of the main theorem.
Finally, in Section \ref{sc:rem}, we give a short remark on nonlinearizations of the
$2$-uniform smoothness and convexity of Banach spaces in connection with curvature bounds in metric geometry.

\begin{acknowledgement}
I would like to express my gratitude to Assaf Naor and Yuval Peres for their valuable comments
on the first version of the paper.
Their suggestions exceedingly improved the presentation of the paper (see Remark \ref{re:Naor} and Proposition \ref{pr:Naor}).
This work was completed while I was visiting Institut f\"ur Angewandte Mathematik, Universit\"at Bonn.
I am grateful to the institute for its hospitality.
\end{acknowledgement}

\section{Nonlinear types}\label{sc:Mark}%%%%%%%%%%%%%%%%%%%%%%%%%%%%%%%%%%%%%%%%%%%%%%

In this section, we recall Rademacher type and cotype of Banach spaces and several extensions of Rademacher type
to nonlinear spaces.
We refer to \cite{LT} and \cite{MS} for basic facts on Rademacher type and cotype.
Throughout the article, we restrict ourselves to the case of $p=2$, i.e.,
we will treat only type $2$ and cotype $2$.

A Banach space $(V,\| \cdot \|)$ is said to have {\it Rademacher type $2$}
if there is a constant $K \ge 1$ such that, for any $N \in \N$ and $\{ v_i \}_{i=1}^N \subset V$, we have
\begin{equation}\label{eq:ty}
\frac{1}{2^N} \sum_{\ve \in \{ -1,1 \}^N} \bigg\| \sum_{i=1}^N \ve_i v_i \bigg\|^2
 \le K^2 \sum_{i=1}^N \| v_i \|^2,
\end{equation}
where $\ve=(\ve_i)_{i=1}^N$.
A fundamental example of a space possessing Rademacher type $2$ is a $2$-uniformly smooth Banach space.
A Banach space $(V,\| \cdot \|)$ is said to be {\it $2$-uniformly smooth}
(or, equivalently, have {\it modulus of smoothness of power type $2$}) if there is a constant $S \ge 1$ such that
\begin{equation}\label{eq:us}
\bigg\| \frac{v+w}{2} \bigg\|^2 \ge \frac{1}{2}\| v \|^2 +\frac{1}{2}\| w \|^2 -\frac{S^2}{4}\| v-w \|^2
\end{equation}
holds for all $v,w \in V$ (see \cite{BCL}).
The infimum of such a constant $S$ is denoted by $S_2(V)$.
For instance, for $2 \le p <\infty$, an $L^p$-space $L^p(Z)$ over an arbitrary measure space $Z$
is $2$-uniformly smooth with $S_2=\sqrt{p-1}$, and hence it has Rademacher type $2$.
Note that, if $V$ is a Hilbert space, then the parallelogram identity yields equality in $(\ref{eq:us})$ with $S=1$.

Rademacher cotype $2$ and the $2$-uniform convexity of a Banach space are defined similarly
by replacing $(\ref{eq:ty})$ and $(\ref{eq:us})$ with
\begin{align}
&\frac{1}{2^N} \sum_{\ve \in \{ -1,1 \}^N} \bigg\| \sum_{i=1}^N \ve_i v_i \bigg\|^2
 \ge \frac{1}{K^2} \sum_{i=1}^N \| v_i \|^2, \label{eq:co}\\
&\bigg\| \frac{v+w}{2} \bigg\|^2 \le \frac{1}{2} \| v \|^2 +\frac{1}{2} \| w \|^2 -\frac{1}{4C^2} \| v-w \|^2, \label{eq:uc}
\end{align}
respectively.
Denote by $C_2(V)$ the least constant $C \ge 1$ satisfying $(\ref{eq:uc})$.
A $2$-uniformly convex Banach space has Rademacher cotype $2$.
In particular, for $1<p \le 2$, $L^p(Z)$ is $2$-uniformly convex with $C_2=1/\sqrt{p-1}$ and has Rademacher cotype $2$.
It is known that also $L^1(Z)$ has Rademacher cotype $2$, though it is not $2$-uniformly convex.

The first nonlinear extension of Rademacher type was given by Enflo.

\begin{definition}\label{df:En}(Enflo type, \cite{E})
A metric space $(X,d)$ is said to have {\it Enflo type $2$}
if there is a constant $K \ge 1$ such that, for any $N \in \N$ and $\{ x_{\ve} \}_{\ve \in \{ -1,1 \}^N} \subset X$,
it holds that
\begin{equation}\label{eq:En}
\sum_{\ve \in \{ -1,1 \}^N} d(x_{\ve},x_{-\ve})^2 \le K^2 \sum_{\ve \sim \ve'} d(x_{\ve},x_{\ve'})^2,
\end{equation}
where $\ve=(\ve_i)_{i=1}^N$ and $\ve \sim \ve'$ holds if $\sum_{i=1}^N|\ve_i-\ve'_i|=2$
(i.e., $\ve$ and $\ve'$ are adjacent).
The least such a constant $K \ge 1$ is denoted by $E_2(X)$.
\end{definition}

By taking $x_{\ve}=\sum_{i=1}^N \ve_i v_i$, we easily see that Enflo type $2$ implies Rademacher type $2$ for Banach spaces.
However, the converse is not known in general.
See \cite{NS} for a partial positive result and \cite{MN2} for related work.

We next recall Markov type introduced by Ball.
As is indicated in its name, we use a Markov chain to define Markov type.
For $N \in \N$, consider a stationary, reversible Markov chain $\{ M_l \}_{l \in \N \cup \{ 0 \}}$
on the state space $\{ 1,2,\ldots,N \}$ with transition probabilities $a_{ij}:=\Prb(M_{l+1}=j \, |\, M_l=i)$.
Namely, if we set $\pi_i:=\Prb(M_0=i)$, then $\{ \pi_i \}_{i=1}^N$ and $A=(a_{ij})_{i,j=1}^N$ satisfy
\begin{equation}\label{eq:A}
0 \le \pi_i \le 1,\quad 0\le a_{ij} \le 1,\quad \sum_{i=1}^N \pi_i =1, \quad
 \sum_{j=1}^N a_{ij}=1, \quad \pi_i a_{ij} =\pi_j a_{ji}
\end{equation}
for all $i,j=1,2,\ldots,N$.
The third and fourth inequalities guarantee the stationariness ($\sum_{i=1}^N \pi_i a_{ij}=\pi_j$)
and the reversibility of the Markov chain $\{ M_l \}_{l \in \N \cup \{ 0 \}}$.

\begin{definition}\label{df:Ma}(Markov type, \cite[Definition 1.3]{B})
A metric space $(X,d)$ is said to have {\it Markov type $2$} if there is a constant $K \ge 1$ such that,
for any $\alpha \in (0,1)$, $N \in \N$, $\{ x_i \}_{i=1}^N \subset X$, $\{ \pi_i \}_{i=1}^N$
and $A=(a_{ij})_{i,j=1}^N$ satisfying $(\ref{eq:A})$, we have
\begin{equation}\label{eq:M1}
(1-\alpha) \sum_{i,j=1}^N \pi_i c_{ij} d(x_i,x_j)^2 \le K^2 \alpha \sum_{i,j=1}^N \pi_i a_{ij} d(x_i,x_j)^2,
\end{equation}
where we set $C=(c_{ij})_{i,j=1}^N=(1-\alpha)(I-\alpha A)^{-1}$ and $I$ stands for the identity matrix.
The infimum of $K \ge 1$ satisfying $(\ref{eq:M1})$ is denoted by $M_2(X)$.
\end{definition}

We remark that Ball's original definition concerns only the case of $\pi_i \equiv N^{-1}$.
The above slightly extended (but equivalent) formulation can be found in \cite{NPSS}.
Note that
\[ C=(1-\alpha)(I-\alpha A)^{-1} =(1-\alpha) \sum_{l=0}^{\infty} \alpha^l A^l. \]
Hence $C=(c_{ij})_{i,j=1}^N$ also satisfies
\[ 0 \le c_{ij} \le 1, \quad \sum_{j=1}^N c_{ij}=1, \quad \pi_i c_{ij} =\pi_j c_{ji} \]
for all $i,j=1,2,\ldots,N$.

We recall some important properties of Markov type.
Markov type has an equivalent form which is more convenient in some circumstances.
For $l \in \N$ and $A=(a_{ij})_{i,j=1}^N$, we set $A^l=(a^{(l)}_{ij})_{i,j=1}^N$.
In particular, $a^{(1)}_{ij} =a_{ij}$.

\begin{theorem}\label{th:eq}{\rm (\cite[Theorem 1.6]{B})}
Let $(X,d)$ be a metric space and assume that there is a constant $K \ge 1$ such that the inequality
\begin{equation}\label{eq:M2}
\sum_{i,j=1}^N \pi_i a^{(l)}_{ij} d(x_i,x_j)^2 \le K^2l \sum_{i,j=1}^N \pi_i a_{ij} d(x_i,x_j)^2
\end{equation}
holds for all $l \in \N$, $N \in \N$, $\{ x_i \}_{i=1}^N \subset X$, $\{ \pi_i \}_{i=1}^N$
and $A=(a_{ij})_{i,j=1}^N$ satisfying $(\ref{eq:A})$.
Then $(X,d)$ has Markov type $2$ with $M_2(X) \le K$.
Conversely, if $(X,d)$ has Markov type $2$, then $(X,d)$ satisfies $(\ref{eq:M2})$ with $K=2\sqrt{e}M_2(X)$.
\end{theorem}

Markov type is known to be strong enough for implying Enflo type.

\begin{proposition}\label{pr:ME}{\rm (\cite[Proposition 1]{NS})}
If a metric space $(X,d)$ has Markov type $2$, then it has Enflo type $2$.
\end{proposition}

To state Ball's theorem which guarantees the usefulness of Markov type,
we need to define Markov cotype of Banach spaces also introduced by Ball.

\begin{definition}\label{df:co}(Markov cotype, \cite[Definition 1.5]{B})
A Banach space $(V,\| \cdot \|)$ is said to have {\it Markov cotype $2$} if there is a constant $K \ge 1$ such that,
for any $\alpha \in (0,1)$, $N \in \N$, $\{ v_i \}_{i=1}^N \subset V$ and $A=(a_{ij})_{i,j=1}^N$
satisfying $(\ref{eq:A})$ with $\pi_i \equiv N^{-1}$, we have
\[ \alpha \sum_{i,j=1}^N a_{ij} \bigg\| \sum_{k=1}^N(c_{ik}-c_{jk})v_k \bigg\|^2
 \le K^2 (1-\alpha) \sum_{i,j=1}^N c_{ij} \| v_i-v_j \|^2, \]
where we set $C=(c_{ij})_{i,j=1}^N=(1-\alpha)(I-\alpha A)^{-1}$.
We denote by $N_2(V)$ the infimum of such a constant $K \ge 1$.
\end{definition}

We remark that Markov cotype is strictly stronger than Rademacher cotype,
for $L^1(Z)$ has Rademacher cotype $2$ and does not have Markov cotype $2$ (see \cite{B}).
It is known that a $2$-uniformly convex Banach space $(V,\| \cdot \|)$ has Markov cotype $2$
with $N_2(V) \le 2C_2(V)$ (\cite[Theorem 4.1]{B}).
For a Lipschitz continuous map $f:X \lra Y$ between metric spaces,
we denote by $\Lip(f)$ its Lipschitz constant, that is,
\[ \Lip(f) := \sup_{x,y \in X,\ x \neq y} \frac{d_Y(f(x),f(y))}{d_X(x,y)}. \]

\begin{theorem}\label{th:Ball}{\rm (\cite[Theorem 1.7]{B})}
Let $(X,d)$ be a metric space having Markov type $2$ and $(V,\| \cdot \|)$ be a reflexive Banach space
having Markov cotype $2$.
Then, for any Lipschitz continuous map $f:Z \lra V$ from a subset $Z \subset X$,
there exists a Lipschitz continuous extension $\tilde{f}:X \lra V$ of $f$ with
\[ \Lip(\tilde{f}) \le 3M_2(X)N_2(V) \Lip(f). \]
In particular, if $(V,\| \cdot \|)$ is a $2$-uniformly convex Banach space, then we have
\[ \Lip(\tilde{f}) \le 6M_2(X)C_2(V) \Lip(f). \]
\end{theorem}

We refer to \cite{BLMN}, \cite{LMN}, \cite{MN1} and \cite{NPSS} for further applications of Markov type.
We end this section with several examples of spaces having Markov type.

\begin{example}\label{ex:Ma}
(i) (Hilbert spaces, \cite[Proposition 1.4]{B})
A Hilbert space $(H,\langle \cdot,\cdot \rangle)$ has Markov type $2$ with $M_2(H)=1$.

(ii) (Products)
For two metric spaces $(X_1,d_1)$ and $(X_2,d_2)$ having Markov type $2$,
let $(X,d)$ be the $l^2$-product of them, that is, $X:=X_1 \times X_2$ and
\[ d\big( (x_1,x_2),(y_1,y_2) \big) :=\{ d_1(x_1,y_1)^2 +d_2(x_2,y_2)^2 \}^{1/2} \]
for $(x_1,x_2),(y_1,y_2) \in X$.
Then $(X,d)$ has Markov type $2$ with
\[ M_2(X) \le \max\{ M_2(X_1),M_2(X_2) \}. \]

(iii) (The bi-Lipschitz equivalence)
Given two metric spaces $(X_1,d_1)$ and $(X_2,d_2)$, if $(X_1,d_1)$ has Markov type $2$ and if
there is a bi-Lipschitz homeomorphism $f:X_1 \lra X_2$, then $(X_2,d_2)$ has Markov type $2$ with
\[ M_2(X_2) \le \Lip(f) \Lip(f^{-1}) M_2(X_1). \]

(iv) (Gromov-Hausdorff limits)
If a sequence of (pointed) metric spaces $\{ (X_i,d_i) \}_{i=1}^{\infty}$ converges to a (pointed) metric space $(X,d)$
in the sense of the (pointed)  Gromov-Hausdorff convergence and if every $(X_i,d_i)$ has Markov type $2$
with $\liminf_{i \to \infty}M_2(X_i) <\infty$, then $(X,d)$ has Markov type $2$ with
\[ M_2(X) \le \liminf_{i \to \infty} M_2(X_i). \]

(v) ($2$-uniformly smooth Banach spaces, \cite[Theorem 1.2]{NPSS})
A $2$-uniformly smooth Banach space $(V,\| \cdot \|)$ has Markov type $2$ with $M_2(V) \le 4S_2(V)$.

(vi) (Trees and hyperbolic groups, \cite[Theorem 1.4, Corollary 1.6]{NPSS})
There exists a universal constant $C_t$ for which every tree $T$ with arbitrary positive edge lengths
has Markov type $2$ with $M_2(T) \le C_t$.
There also exists a universal constant $C_h$ such that every $\delta$-hyperbolic group has Markov type $2$
with $M_2 \le C_h(1+\delta)$.
More precisely, we fix a presentation of the group and consider its Cayley graph $G$ equipped with the word metric.
If $G$ is $\delta$-hyperbolic as a metric space, then it has Markov type $2$ with $M_2(G) \le C_h(1+\delta)$.
Naor et al.\ have obtained an estimate for general $\delta$-hyperbolic metric spaces, and it implies the above results.

(vii) (Riemannian manifolds with pinched negative sectional curvature, \cite[Theorem 1.7]{NPSS})
An $n$-dimensional, complete and simply connected Riemannian manifold $M$ has Markov type $2$
if its sectional curvature takes values in $[-R,-r]$ for some $R>r>0$.
Then $M_2(M)$ is estimated from above by using $n$, $r$ and $R$.

(viii) (Laakso graphs, \cite[Proposition 7.1]{NPSS})
The Laakso graphs (\cite{La}) have Markov type $2$.
\end{example}

\section{Alexandrov spaces of nonnegative curvature}\label{sc:Alex}%%%%%%%%%%%%%%%%%%%%%%%%%%%%%%%%%%%

In this section, we recall the definition of Alexandrov spaces of nonnegative curvature.
We refer to \cite{BGP} and \cite{BBI} as standard references.

A metric space $(X,d)$ is said to be {\it geodesic} if every two points $x,y \in X$ can be connected by
a curve $\gamma:[0,1] \lra X$ from $x$ to $y$ with $\length(\gamma)=d(x,y)$.
A rectifiable curve $\gamma:[0,1] \lra X$ is called a {\it geodesic} if it is locally minimizing and has a constant speed.
A geodesic $\gamma:[0,1] \lra X$ is said to be {\it minimal} if it satisfies $\length(\gamma)=d(\gamma(0),\gamma(1))$.

\begin{definition}
A geodesic metric space $(X,d)$ is called an {\it Alexandrov space of nonnegative curvature}
if, for all three points $x,y,z \in X$ and any minimal geodesic $\gamma:[0,1] \lra X$ between $y$ and $z$,
we have
\[ d \bigg( x,\gamma \bigg( \frac{1}{2} \bigg) \bigg)^2
 \ge \frac{1}{2}d(x,y)^2 +\frac{1}{2}d(x,z)^2 -\frac{1}{4}d(y,z)^2. \]
\end{definition}

We recall some examples.
Every example gives a new class of metric spaces having Markov type $2$.

\begin{example}
(i) A complete Riemannian manifold is an Alexandrov space of nonnegative curvature
if and only if its sectional curvature is nonnegative everywhere.
In particular, spheres, tori and symmetric spaces of compact type are Alexandrov spaces of nonnegative curvaure.

(ii) For a compact convex domain $\Omega \subset \R^n$,
let $X=\partial \Omega$ equip the length metric $d$ induced from the standard metric of $\R^n$.
Then $(X,d)$ is an Alexandrov space of nonnegative curvature.

(iii)
Let $(M,g)$ be a Riemannian manifold of nonnegative sectional curvature and $G$ be a compact group acting on $M$ by isometries.
Then the quotient space $M/G$ equipped with the quotient metric is an Alexandrov space of nonnegative curvature.
\end{example}

There is a rich and deep theory on the geometry and analysis on Alexandrov spaces,
but all we need in the present paper is the following characterization due to Sturm
(see also \cite[Proposition 3.2]{LS}).

\begin{theorem}\label{th:S}{\rm (\cite[Theorem 1.4]{S})}
A geodesic metric space $(X,d)$ is an Alexandrov space of nonnegative curvature if and only if,
for any $N \in \N$, $\{ x_i \}_{i=1}^N \subset X$, $y \in X$ and $\{ a_i \}_{i=1}^N \subset [0,1]$ with $\sum_{i=1}^N a_i=1$,
we have
\begin{equation}\label{eq:bc*}
\sum_{i,j=1}^N a_i a_j \{ d(x_i,x_j)^2-d(x_i,y)^2-d(x_j,y)^2 \} \le 0.
\end{equation}
\end{theorem}

More probabilistically speaking, the inequality $(\ref{eq:bc*})$ says that,
for any finitely supported $X$-valued random variable $Z$ and its independent copy $\widetilde{Z}$,
\[ \E[d(Z,\widetilde{Z})^2] \le 2\E[d(Z,y)^2] \]
holds for all $y \in X$.

We also remark that the inequality $(\ref{eq:bc*})$ corresponds to the following fact
in a Hilbert space $(H,\langle \cdot,\cdot \rangle)$.
For any $N \in \N$, $\{ v_i \}_{i=1}^N \subset H$ and $\{ a_i \}_{i=1}^N \subset [0,1]$ with $\sum_{i=1}^N a_i=1$,
\begin{align*}
&\sum_{i,j=1}^N a_i a_j \{ \| v_i-v_j \|^2 -\| v_i-w \|^2 -\| v_j-w \|^2 \} \\
&= 2\sum_{i,j=1}^N a_i a_j \langle v_i-w,w-v_j \rangle
 = 2\bigg\langle \bigg( \sum_{i=1}^N a_i v_i \bigg) -w,w- \bigg( \sum_{j=1}^N a_j v_j \bigg) \bigg\rangle \\
&= -2\bigg\| \bigg( \sum_{i=1}^N a_i v_i \bigg) -w \bigg\|^2 \le 0
\end{align*}
holds for all $w \in H$.

\section{Markov type of Alexandrov spaces}\label{sc:thm}%%%%%%%%%%%%%%%%%%%%%%%%%%%%%%%%%%%%%%%%%%%%%%%%%%%

In this section, we prove our main theorem.
Throughout the section, let $(X,d)$ be an Alexandrov space of nonnegative curvature,
and fix $N \in \N$, $\{ x_i \}_{i=1}^N \subset X$, $\{ \pi_i \}_{i=1}^N$
and $A=(a_{ij})_{i,j=1}^N$ satisfying $(\ref{eq:A})$.
For $1 \le i,j \le N$ and $l \in \N$, set $d_{ij}:=d(x_i,x_j)$ and
\[ \cE(l):=\sum_{i,j=1}^N \pi_i a^{(l)}_{ij}d_{ij}^2. \]
Recall the notation $A^l=(a^{(l)}_{ij})_{i,j=1}^N$ and that $(\ref{eq:A})$ implies
\[ 0 \le a^{(l)}_{ij} \le 1, \quad \sum_{j=1}^N a^{(l)}_{ij}=1, \quad
 \pi_i a^{(l)}_{ij}=\pi_j a^{(l)}_{ji} \]
for all $i,j=1,2,\ldots,N$.

\begin{lemma}\label{lm:half}
For any $l \in \N$, we have $\cE(2l) \le 2\cE(l)$.
\end{lemma}

\begin{proof}
We calculate
\begin{align*}
&\cE(2l) =\sum_{i,j,k=1}^N \pi_i a^{(l)}_{ik} a^{(l)}_{kj} d_{ij}^2 \\
&= \sum_{i,j,k=1}^N \pi_k a^{(l)}_{ki} a^{(l)}_{kj} (d_{ki}^2+d_{kj}^2+d_{ij}^2-d_{ki}^2-d_{kj}^2) \\
&= \sum_{i,j,k=1}^N \pi_k a^{(l)}_{ki} a^{(l)}_{kj} (d_{ki}^2+d_{kj}^2)
 +\sum_{k=1}^N \pi_k \bigg\{ \sum_{i,j=1}^N a^{(l)}_{ki} a^{(l)}_{kj} (d_{ij}^2-d_{ki}^2-d_{kj}^2) \bigg\}.
\end{align*}
Since $\sum_{i=1}^N a^{(l)}_{ki}=1$, we have
\[ \sum_{i,j,k=1}^N \pi_k a^{(l)}_{ki} a^{(l)}_{kj} (d_{ki}^2+d_{kj}^2)
 =\sum_{i,k=1}^N \pi_k a^{(l)}_{ki} d_{ki}^2 +\sum_{j,k=1}^N \pi_k a^{(l)}_{kj} d_{kj}^2
 =2\cE(l). \]
Moreover, applying Theorem \ref{th:S} with $a_i=a^{(l)}_{ki}$ and $y=x_k$, we obtain
\[ \sum_{i,j=1}^N a^{(l)}_{ki} a^{(l)}_{kj} (d_{ij}^2-d_{ki}^2-d_{kj}^2) \le 0 \]
for any $k$.
This completes the proof.
$\qedd$
\end{proof}

\begin{theorem}\label{th:main}
Let $(X,d)$ be an Alexandrov space of nonnegative curvature.
Then $(X,d)$ has Markov type $2$ with $M_2(X) \le 1+\sqrt{2}$.
More precisely, $(X,d)$ satisfies the inequality $(\ref{eq:M2})$ with $K=1+\sqrt{2}$.
\end{theorem}

\begin{proof}
We will prove the theorem by induction.
Note that the triangle inequality implies $\sqrt{\cE(l+m)} \le \sqrt{\cE(l)}+\sqrt{\cE(m)}$ for all $l,m \in \N$,
and that Lemma \ref{lm:half} yields $\cE(2^n) \le 2^n\cE(1)$ for all $n \in \N$.
Assume that $\cE(l) \le (1+\sqrt{2})^2l\cE(1)$ holds for all $1 \le l \le 2^n$ for fixed $n \in \N$.
Then, for $2^n+1 \le l \le 2^{n+1}$, take $t \in (0,1]$ with $l=(1+t)2^n$ and observe
\begin{align*}
\sqrt{\cE(l)} &\le \sqrt{\cE(2^n)}+\sqrt{\cE(t2^n)} \le \sqrt{2^n \cE(1)}+(1+\sqrt{2})\sqrt{t2^n\cE(1)} \\
&= \{ 1+(1+\sqrt{2})\sqrt{t} \} \sqrt{2^n\cE(1)} \le (1+\sqrt{2})\sqrt{1+t}\sqrt{2^n\cE(1)} \\
&=(1+\sqrt{2})\sqrt{l\cE(1)}.
\end{align*}
Here the fourth implication follows from the fact that the function
\[ f(t)=(1+\sqrt{2})(\sqrt{1+t}-\sqrt{t}) \]
is decreasing in $t \in (0,1]$ and $f(1)=1$.
$\qedd$
\end{proof}

\begin{remark}\label{re:Naor}
The author's original proof used Lemma \ref{lm:half} as well as the inequality
\[ \alpha^{2l} \cE(2l) +2\alpha^{2l+1} \cE(2l+1) +\alpha^{2l+2} \cE(2l+2)
 \le 2(1+\alpha)\alpha^l \{ \alpha^l \cE(l) +\alpha^{l+1} \cE(l+1) \} \]
for $l \in \N$ and $\alpha \in (0,1)$,
and then we obtain by calculation that $M_2(X) \le \sqrt{6}$.
After the first version of this paper was completed, the author learned
the above simpler, improved proof from A.~Naor and Y.~Peres.
\end{remark}

We have two corollaries by virtue of Proposition \ref{pr:ME} and Theorem \ref{th:Ball}.

\begin{corollary}\label{cr:En}
Let $(X,d)$ be an Alexandrov space of nonnegative curvature.
Then $(X,d)$ has Enflo type $2$.
\end{corollary}

\begin{corollary}\label{cr:Lip}
Let $(X,d)$ be an Alexandrov space of nonnegative curvature
and $(V,\| \cdot \|)$ be a reflexive Banach space having Markov cotype $2$.
Then, for any Lipschitz continuous map $f:Z \lra V$ from a subset $Z \subset X$,
there exists a Lipschitz continuous extension $\tilde{f}:X \lra V$ of $f$ with
\[ \Lip(\tilde{f}) \le 3(1+\sqrt{2})N_2(V) \Lip(f). \]
In particular, if $(V,\| \cdot \|)$ is $2$-uniformly convex, then we have
\[ \Lip(\tilde{f}) \le 6(1+\sqrt{2})C_2(V) \Lip(f). \]
\end{corollary}

We mention that our bound of the ratio of Lipschitz constants is independent of the dimension of $X$.
Compare this with \cite[Theorem 1.6]{LN}.

\section{Additional remarks}\label{sc:rem}%%%%%%%%%%%%%%%%%%%%%%%%%%%%%%%%%%%%%%%%%%%

This section is devoted to a short remark toward a nonlinearization of the $2$-uniform smoothness (and convexity).
As we have already seen in $(\ref{eq:us})$, the $2$-uniform smoothness of a Banach space
is defined by using the inequality
\begin{equation}\label{eq:usg}
\bigg\| \frac{v+w}{2} \bigg\|^2 \ge \frac{1}{2}\| v \|^2 +\frac{1}{2}\| w \|^2 -\frac{S^2}{4}\| v-w \|^2.
\end{equation}
By replacing $v$ and $w$ with $w+v$ and $w-v$, this inequality is rewritten as
\begin{equation}\label{eq:usl}
\bigg\| \frac{v+w}{2} \bigg\|^2 \le \frac{S^2}{2}\| v \|^2 +\frac{1}{2}\| w \|^2 -\frac{1}{4}\| v-w \|^2.
\end{equation}

Natural generalizations of $(\ref{eq:usg})$ and $(\ref{eq:usl})$ would be the following:
Let $(X,d)$ be a geodesic metric space.
For any three points $x,y,z \in X$ and minimal geodesic $\gamma:[0,1] \lra X$ from $y$ to $z$, we have
\begin{equation}\label{eq:usA}
d\bigg( x,\gamma \bigg( \frac{1}{2} \bigg) \bigg)^2
\ge \frac{1}{2}d(x,y)^2 +\frac{1}{2}d(x,z)^2 -\frac{S^2}{4}d(y,z)^2
\end{equation}
or
\begin{equation}\label{eq:usC}
d\bigg( x,\gamma \bigg( \frac{1}{2} \bigg) \bigg)^2
\le \frac{S^2}{2}d(x,y)^2 +\frac{1}{2}d(x,z)^2 -\frac{1}{4}d(y,z)^2.
\end{equation}
We will say that a geodesic metric space $(X,d)$ {\it satisfies $(\ref{eq:usA})$} (or $(\ref{eq:usC})$) if $(\ref{eq:usA})$
(or $(\ref{eq:usC})$) holds for all $x,y,z \in X$ and all minimal geodesic $\gamma:[0,1] \lra X$ from $y$ to $z$.
On one hand, the inequality $(\ref{eq:usA})$ generalizes the nonnegatively curved property in the sense of
Alexandrov which corresponds to the case of $S=1$ (see Section \ref{sc:Alex}).
On the other hand, the inequality $(\ref{eq:usC})$ extends the $\CAT(0)$-inequality
which amounts to the case of $S=1$ (cf.\ \cite{BH}).
This is a reason why both negatively and positively curved spaces have Markov type $2$.
Compare Example \ref{ex:Ma} and Theorem \ref{th:main}.

We mention that we can also regard $(\ref{eq:usg})$ as an upper curvature bound of the unit sphere (see \cite{O1}),
and that the reverse inequality of $(\ref{eq:usA})$ (a generalized $2$-uniform convexity) has been studied in \cite{O2}.

As an application of the inequality $(\ref{eq:usC})$, we give an example of a nonlinear and non-Riemannian
(in other words, Finslerian) space possessing Enflo type $2$.
We first prove a lemma.

\begin{lemma}\label{lm:CAT}
Let a geodesic metric space $(X,d)$ satisfy $(\ref{eq:usC})$.
Then, for any four points $w,x,y,z \in X$, we have
\begin{equation}\label{eq:four}
d(w,y)^2+d(x,z)^2 \le S^2 \{ d(w,x)^2+d(y,z)^2 \} +d(w,z)^2+d(y,x)^2.
\end{equation}
\end{lemma}

\begin{proof}
Take a minimal geodesic $\gamma:[0,1] \lra X$ between $x$ and $z$.
Then $(\ref{eq:usC})$ yields that
\begin{align*}
d\bigg( w,\gamma \bigg( \frac{1}{2} \bigg) \bigg)^2
&\le \frac{S^2}{2}d(w,x)^2 +\frac{1}{2}d(w,z)^2 -\frac{1}{4}d(x,z)^2, \\
d\bigg( y,\gamma \bigg( \frac{1}{2} \bigg) \bigg)^2
&\le \frac{S^2}{2}d(y,z)^2 +\frac{1}{2}d(y,x)^2 -\frac{1}{4}d(x,z)^2.
\end{align*}
Thus we see
\begin{align*}
d(w,y)^2 &\le \bigg\{ d\bigg( w,\gamma \bigg( \frac{1}{2} \bigg) \bigg)
 +d\bigg( \gamma \bigg( \frac{1}{2} \bigg),y \bigg) \bigg\}^2 \\
&\le 2\bigg\{ d\bigg( w,\gamma \bigg( \frac{1}{2} \bigg) \bigg)^2
 +d\bigg( \gamma \bigg( \frac{1}{2} \bigg),y \bigg)^2 \bigg\} \\
&\le S^2 \{ d(w,x)^2+d(y,z)^2 \} +d(w,z)^2+d(y,x)^2-d(x,z)^2.
\end{align*}
This is the required inequality.
$\qedd$
\end{proof}

\begin{proposition}\label{pr:CAT}
If a geodesic metric space $(X,d)$ satisfies $(\ref{eq:usC})$, then it has Enflo type $2$ with $E_2(X) \le S$.
In particular, a $\CAT(0)$-space $(X,d)$ has Enflo type $2$ with $E_2(X)=1$, and
a $2$-uniformly smooth Banach space $(V,\| \cdot \|)$ has Enflo type $2$ with $E_2(V) \le S_2(V)$.
\end{proposition}

\begin{proof}
We shall prove by induction in $N \in \N$.
In the case of $N=1$, for any $\{ x_1,x_{-1} \} \subset X$, we immediately see
\[ d(x_1,x_{-1})^2+d(x_{-1},x_1)^2 \le S^2 \{ d(x_1,x_{-1})^2+d(x_{-1},x_1)^2 \}. \]
Fix $N \ge 2$ and suppose that, for any $\{ x_{\delta} \}_{\delta \in \{ -1,1 \}^{N-1}} \subset X$,
it holds that
\[ \sum_{\delta \in \{ -1,1 \}^{N-1}} d(x_{\delta},x_{-\delta})^2
 \le S^2 \sum_{\delta \sim \delta'} d(x_{\delta},x_{\delta'})^2, \]
where $\delta=(\delta_i)_{i=1}^{N-1}$ and $\delta \sim \delta'$ holds if $\sum_{i=1}^{N-1} |\delta_i-\delta'_i|=2$.
Now we choose an arbitrary $\{ x_{\ve} \}_{\ve \in \{ -1,1 \}^N} \subset X$.
For each $\delta \in \{ -1,1 \}^{N-1}$, Lemma \ref{lm:CAT} implies
\begin{align*}
&d(x_{(\delta,1)},x_{(-\delta,-1)})^2 +d(x_{(\delta,-1)},x_{(-\delta,1)})^2 \\
&\le S^2 \{ d(x_{(\delta,1)},x_{(\delta,-1)})^2 +d(x_{(-\delta,-1)},x_{(-\delta,1)})^2 \}
 +d(x_{(\delta,1)},x_{(-\delta,1)})^2 +d(x_{(-\delta,-1)},x_{(\delta,-1)})^2.
\end{align*}
Summing up this inequality in $\delta \in \{ -1,1 \}^{N-1}$, we have
\begin{align*}
\sum_{\ve \in \{ -1,1 \}^N} d(x_{\ve},x_{-\ve})^2
&\le S^2 \sum_{\delta \in \{ -1,1 \}^{N-1}} \{ d(x_{(\delta,1)},x_{(\delta,-1)})^2+d(x_{(\delta,-1)},x_{(\delta,1)})^2 \} \\
&\qquad +\sum_{\delta \in \{ -1,1 \}^{N-1}} \{ d(x_{(\delta,1)},x_{(-\delta,1)})^2+d(x_{(\delta,-1)},x_{(-\delta,-1)})^2 \}.
\end{align*}
By our assumption, the second term in the right-hand side is estimated as
\begin{align*}
&\sum_{\delta \in \{ -1,1 \}^{N-1}} \{ d(x_{(\delta,1)},x_{(-\delta,1)})^2+d(x_{(\delta,-1)},x_{(-\delta,-1)})^2 \} \\
&\qquad \le S^2 \sum_{\delta \sim \delta'} \{ d(x_{(\delta,1)},x_{(\delta',1)})^2+d(x_{(\delta,-1)},x_{(\delta',-1)})^2 \}.
\end{align*}
Therefore we obtain
\[ \sum_{\ve \in \{ -1,1 \}^N} d(x_{\ve},x_{-\ve})^2 \le S^2 \sum_{\ve \sim \ve'}d(x_{\ve},x_{\ve'})^2. \]
This completes the proof.
$\qedd$
\end{proof}

The following observation (in connection with \cite{FS}) is due to A.~Naor.

\begin{proposition}\label{pr:Naor}
Let $(X,d)$ be a metric space satisfying the Ptolemy inequality, that is,
for all four points $w,x,y,z \in X$, we have
\begin{equation}\label{eq:Pto}
d(w,y) \cdot d(x,z) \le d(w,x) \cdot d(y,z) +d(w,z) \cdot d(y,x).
\end{equation}
Then $(X,d)$ satisfies the inequality $(\ref{eq:four})$ with $S=\sqrt{3}$.
In particular, $(X,d)$ has Enflo type $2$ with $E_2(X) \le \sqrt{3}$.
\end{proposition}

\begin{proof}
For any four points $w,x,y,z \in X$, the Ptolemy inequality $(\ref{eq:Pto})$ yields
\begin{align*}
&d(w,y)^2 +d(x,z)^2 -\{ d(w,y)-d(x,z) \}^2 =2d(w,y) \cdot d(x,z) \\
&\le 2d(w,x) \cdot d(y,z) +2d(w,z) \cdot d(y,x) \\
&\le d(w,x)^2 +d(y,z)^2 +d(w,z)^2 +d(y,x)^2.
\end{align*}
It follows from the triangle inequality that
\[ \{ d(w,y)-d(x,z) \}^2 \le \{ d(w,x)+d(y,z) \}^2 \le 2d(w,x)^2+2d(y,z)^2. \]
Therefore we obtain
\begin{align*}
&d(w,y)^2+d(x,z)^2 \\
&\le d(w,x)^2+d(y,z)^2+d(w,z)^2+d(y,x)^2 +\{ d(w,y)-d(x,z) \}^2 \\
&\le 3\{ d(w,x)^2+d(y,z)^2 \} +d(w,z)^2+d(y,x)^2.
\end{align*}
The proof of Proposition \ref{pr:CAT} shows that this inequality derives Enflo type $2$.
$\qedd$
\end{proof}


\begin{thebibliography}{BLMN}%%%%%%%%%%%%%%%%%%%%%%%%%%%%%%%%%%%%%%%%%%%%%%%

\bibitem[B]{B}
K.~Ball, \textit{Markov chains, Riesz transforms and Lipschitz maps},
Geom.\ Funct.\ Anal.\ \textbf{2} (1992), 137--172.

\bibitem[BCL]{BCL}
K.~Ball, E.~A.~Carlen and E.~H.~Lieb,
{\it Sharp uniform convexity and smoothness inequalities for trace norms},
Invent.\ Math.\ {\bf 115} (1994), 463--482.

\bibitem[BLMN]{BLMN}
Y.~Bartal, N.~Linial, M.~Mendel and A.~Naor, {\it On metric Ramsey-type phenomena},
Ann.\ of Math.\ (2) {\bf 162} (2005), 643--709.

\bibitem[BMW]{BMW}
J.~Bourgain, V.~Milman and H.~Wolfson, \textit{On type of metric spaces},
Trans.\ Amer.\ Math.\ Soc.\ \textbf{294} (1986), 295--317.

\bibitem[BH]{BH}
M.~R.~Bridson and A.~Haefliger, Metric spaces of non-positive curvature,
Springer-Verlag, Berlin, 1999.

\bibitem[BBI]{BBI}
D.~Burago, Yu.~Burago and S.~Ivanov, A course in metric geometry,
American Mathematical Society, Providence, RI, 2001.

\bibitem[BGP]{BGP}
Yu.~Burago, M.~Gromov and G.~Perel'man,
\textit{A.\ D.\ Alexandrov spaces with curvatures bounded below},
Russian Math.\ Surveys \textbf{47} (1992), 1--58.

\bibitem[E]{E}
P.~Enflo, \textit{On infinite-dimensional topological groups},
S\'eminaire sur la G\'eom\'etrie des Espaces de Banach (1977--1978), Exp.\ No.\ 10--11, 11 pp.,
\'Ecole Polytech.\ Palaiseau, 1978.

\bibitem[FS]{FS}
T.~Foertsch and V.~Schroeder, \textit{Hyperbolicity, $\CAT(-1)$-spaces and the Ptolemy inequality},
to appear in Math.\ Ann.

\bibitem[La]{La}
T.~J.~Laakso, {\it Ahlfors $Q$-regular spaces with arbitrary $Q>1$ admitting weak Poincar\'e inequality},
Geom.\ Funct.\ Anal.\ {\bf 10} (2000), 111--123.
Erratum, ibid. {\bf 12} (2002), 650.

\bibitem[LPS]{LPS}
U.~Lang, B.~Pavlovi\'c and V.~Schroeder, \textit{Extensions of Lipschitz maps into Hadamard spaces},
Geom.\ Funct.\ Anal.\ \textbf{10} (2000), 1527--1553.

\bibitem[LS]{LS}
U.~Lang and V.~Schroeder, \textit{Kirszbraun's theorem and metric spaces of bounded curvature},
Geom.\ Funct.\ Anal.\ \textbf{7} (1997), 535--560.

\bibitem[LN]{LN}
J.~R.~Lee and A.~Naor, \textit{Extending Lipschitz functions via random metric partitions},
Invent.\ Math.\ \textbf{160} (2005), 59--95.

\bibitem[LT]{LT}
J.~Lindenstrauss and L.~Tzafriri, Classical Banach spaces II,
Springer-Verlag, Berlin-New York, 1979.

\bibitem[LMN]{LMN}
N.~Linial, A.~Magen and A.~Naor, {\it Girth and Euclidean distortion},
Geom.\ Funct.\ Anal.\ {\bf 12} (2002), 380--394.

\bibitem[MN1]{MN1}
M.~Mendel and A.~Naor, {\it Some applications of Ball's extension theorem},
Proc.\ Amer.\ Math.\ Soc.\ {\bf 134} (2006), 2577--2584.

\bibitem[MN2]{MN2}
M.~Mendel and A.~Naor, \textit{Scaled Enflo type is equivalent to Rademacher type},
Bull.\ London Math.\ Soc.\ {\bf 39} (2007), 493--498.

\bibitem[MN3]{MN3}
M.~Mendel and A.~Naor, \textit{Metric cotype}, Ann.\ of Math.\ {\bf 168} (2008), 247--298.

\bibitem[MS]{MS}
V.~D.~Milman and G.~Schechtman, Asymptotic theory of finite-dimensional normed spaces,
with an appendix by M.~Gromov, Lecture Notes in Mathematics \textbf{1200},
Springer-Verlag, Berlin, 1986.

\bibitem[NPSS]{NPSS}
A.~Naor, Y.~Peres, O.~Schramm and S.~Sheffield,
\textit{Markov chains in smooth Banach spaces and Gromov hyperbolic metric spaces},
Duke Math.\ J.\ {\bf 134} (2006), 165--197.

\bibitem[NS]{NS}
A.~Naor and G.~Schechtman, \textit{Remarks on non linear type and Pisier's inequality},
J.\ Reine Angew.\ Math.\ \textbf{552} (2002), 213--236.

\bibitem[O1]{O1}
S.~Ohta, \textit{Regularity of harmonic functions in Cheeger-type Sobolev spaces},
Ann.\ Global Anal.\ Geom.\ \textbf{26} (2004), 397--410.

\bibitem[O2]{O2}
S.~Ohta, \textit{Convexities of metric spaces},
Geom.\ Dedicata {\bf 125} (2007), 225--250.

\bibitem[S]{S}
K.-T.~Sturm, \textit{Metric spaces of lower bounded curvature},
Exposition.\ Math.\ \textbf{17} (1999), 35--47.

\end{thebibliography}
\end{document}